\documentclass[11pt]{amsart}
\usepackage{amsmath}
\usepackage{amsfonts}

\usepackage{amscd}
\usepackage{amsthm}
\usepackage{amssymb} \usepackage{latexsym}
\usepackage{eufrak}
\usepackage{euscript}
\usepackage{epsfig}
\usepackage{graphics}
\usepackage{array}
\usepackage{enumerate}
\usepackage{dsfont}
\usepackage{color}
\usepackage{wasysym}
\usepackage{hyperref}
\usepackage{pdfsync}

\newcommand{\bel}[1]{\begin{equation}\label{#1}}

\newcommand{\be}{\begin{equation}}

\newcommand{\ba}{\begin{eqnarray}}
\newcommand{\ea}{\end{eqnarray}}

\newcommand{\qe}{\end{equation}}
\newcommand{\R}{{\mathbb R}}

\newcommand{\CDE}{\mathrm{CDE}}
\newcommand{\CCD}{\mathrm{CD}}

\newcommand{\de}{{\delta}}

\newcommand{\Hmm}[1]{\leavevmode{\marginpar{\tiny%
$\hbox to 0mm{\hspace*{-0.5mm}$\leftarrow$\hss}%
\vcenter{\vrule depth 0.1mm height 0.1mm width \the\marginparwidth}%
\hbox to
0mm{\hss$\rightarrow$\hspace*{-0.5mm}}$\\\relax\raggedright #1}}}

\theoremstyle{theorem}
\newtheorem{thm}{Theorem}[section]
\newtheorem{prop}[thm]{Proposition}
\theoremstyle{example}
\newtheorem{example}[thm]{Example}
\theoremstyle{corollary}

\theoremstyle{lemma}
\newtheorem{lemma}[thm]{Lemma}
\theoremstyle{definition}
\newtheorem{defi}[thm]{Definition}
\theoremstyle{proof}
\theoremstyle{remark}
\newtheorem{rem}[thm]{Remark}

\begin{document}

\title[Stochastic completeness for graphs]{Stochastic completeness for graphs with curvature dimension conditions}
\author{Bobo Hua}
\email{bobohua@fudan.edu.cn}
\address{School of Mathematical Sciences, LMNS, Fudan University, Shanghai 200433, China; Shanghai Center for Mathematical Sciences, Fudan University, Shanghai 200433, China
}

\author{Yong Lin}
\email{linyong01@ruc.edu.cn}
\address{Department of Mathematics,Information School,
Renmin University of China,
Beijing 100872, China}

\begin{abstract}We prove pointwise gradient bounds for heat semigroups associated to general (possibly unbounded) Laplacians on infinite graphs satisfying the curvature dimension condition $CD(K,\infty)$. Using gradient bounds, we show stochastic completeness for graphs satisfying the curvature dimension condition.
\end{abstract}
\maketitle
\section{Introduction and main results}\label{s:introduction}
Let $M$ be a complete, noncompact Riemannian manifold without boundary. It is called stochastically complete if
\begin{equation}\label{d:stochastic}\int_{M}p_t(x,y)d\mathrm{vol}(y)=1,\quad \quad \forall\ t>0, x\in M,\end{equation} where $p_t(\cdot,\cdot)$ is the (minimal) heat kernel on $M.$
Yau \cite {Yau78} first proved that any complete Riemannian manifold with a uniform lower bound of Ricci curvature is stochastically complete. Karp and Li \cite{KarpLi} showed the stochastic completeness in terms of the following volume growth property:  \begin{equation}\label{e:volume criterion}\mathrm{vol}(B_r(x))\leq Ce^{cr^2},\quad\quad\mathrm{some}\ x\in M, \forall\ r> 0,\end{equation} where $\mathrm{vol}(B_r(x))$ is the volume of the geodesic ball of radius $r$ and centered at $x$. Varopoulos \cite{Varopoulos}, Li \cite{Li84} and Hsu \cite{Hsu89} extended Yau's result to Riemannian manifolds with general conditions on Ricci curvature. So far, the optimal volume growth condition for stochastic completeness was given by Grigor'yan \cite{Grigoryan86}. We refer to \cite{Grigoryan99} for the literature on stochastic completeness of
Riemannian manifolds. These results have been generalized to a quite general setting, namely, regular strongly local Dirichlet forms by Sturm \cite{Sturm94}.

Compared to local operators, graphs (discrete metric measure spaces) are nonlocal in nature and can be regarded as regular Dirichlet forms associated to jump processes. A general Markov semigroup is called a diffusion semigroup if chain rules hold for the associated infinitesimal generator, see Bakry, Gentil and Ledoux \cite[Definition~1.11.1]{BakryGentilLedoux}, which is a property related to the locality of the generator. As a common point of view to many graph analysts, the absence of chain rules for discrete Laplacians is the main difficulty for the analysis on graphs. This causes many problems and various interesting phenomena emerge on graphs. A graph is called \emph{stochastically complete} (or conservative) if an equation similar to \eqref{d:stochastic} holds for the continuous time heat kernel, see Definition \ref{d:stochastic complete graph}. The stochastic completeness of graphs has been thoroughly studied by many authors \cite{DodziukMathai06,Dodziuk06,Wojciechowski08,Weber10,Wojciechowski1,Huang10,KellerLenz10,Wojciechowski2,KellerLenz12,GrigoryanHuangMasamune12,MasamuneUemuraWang12,KellerLenzWojciechowski}. In particular, the volume criterion \eqref{e:volume criterion} with respect to the graph distance is no longer true for unbounded Laplacians on graphs, see \cite{Wojciechowski2}. This can be circumvented by using intrinsic metrics introduced by Frank, Lenz and Wingert \cite{FrankLenzWingert12}, see e.g. \cite{GrigoryanHuangMasamune12,Folz14,Huang14}.

Gradient bounds of heat semigroups can be used to prove stochastic completeness. Nowadays, the so-called $\Gamma$-calculus has been well developed in the framework of general Markov semigroups where $\Gamma$ is called the ``carr\'e du champ" operator, see \cite[Definition~1.4.2]{BakryGentilLedoux}. Given a smooth function $f$ on a Riemannian manifold, $\Gamma(f)$ stands for $|\nabla f|^2,$ see Section~\ref{s:graphs} for the definition on graphs. Heuristically, on a Riemannian manifold $M$ if one can show the gradient bound for the heat semigroup
\begin{equation}\label{e:gradient bounds}\Gamma(P_t f)\leq P_t(\Gamma (f)),\quad\quad \forall f\in C_0^{\infty}(M),\end{equation} where $P_t=e^{t\Delta_M}$ is the heat semigroup induced by the Laplace-Beltrami operator $\Delta_M$ and $C_0^\infty(M)$ is the space of compactly supported smooth functions on $M$, then the stochastic completeness follows from approximating the constant function $\mathds{1}$ by compactly supported smooth functions. This approach has been systematically generalized to Markov diffusion semigroups, i.e. local operators, see \cite{BakryGentilLedoux}. In this paper, we closely follow this strategy and prove the stochastic completeness for the non-diffusion case, i.e. graphs. This shows that the gradient-bound approach works even in nonlocal setting.

We introduce the setting of graphs and refer to Section~\ref{s:graphs} for details.
Let $(V,E)$ be a connected, undirected, (combinatorial) infinite graph with the set of vertices $V$ and the set of edges $E.$  We say $x,y\in V$ are neighbors, denoted by $x\sim y,$ if $(x,y)\in E.$ The graph is called \emph{locally finite} if each vertex has finitely many neighbors. In this paper, we only consider locally finite graphs. We assign a weight $m$ to each vertex, $m: V\to (0,\infty),$ and a weight $\mu$ to each edge, $$\mu:E\to (0,\infty), E\ni (x,y)\mapsto \mu_{xy},$$ and refer to the quadruple $G=(V,E,m,\mu)$ as a \emph{weighted graph}. We denote by $$C_0(V):=\{f:V\to\R|\ \{x\in V| f(x)\neq 0\}\ \mathrm{is\ of\ finite\ cardinality}\}$$ the set of finitely supported functions on $V$ and by $\ell^p(V,m),$ $p\in [1,\infty],$ the $\ell^p$ spaces of functions on $V$ with respect to the measure $m.$

For any weighted graph $G=(V,E,m,\mu)$, it associates with a Dirichlet form with respect to the Hilbert space $\ell^2(V,m)$ corresponding to the Dirichlet boundary condition,
\begin{eqnarray}\label{e:dirichlet form}Q^{(D)}:&&D(Q^{(D)})\times D(Q^{(D)})\to \R \nonumber\\ && (f,g)\mapsto \frac12\sum_{x\sim y}\mu_{xy}(f(y)-f(x))(g(y)-g(x)),\end{eqnarray} where the form domain $D(Q^{(D)})$ is defined as the completion of $C_0(V)$ under the norm $\|\cdot\|_Q$ given by $$\|f\|^2_{Q}=\|f\|_{\ell^2(V,m)}^2+\frac12\sum_{x\sim y}\mu_{xy}(f(y)-f(x))^2, \ \ \forall f\in C_0(V),$$ see Keller and Lenz \cite{KellerLenz12}. For the Dirichlet form $Q^{(D)},$ its (infinitesimal) generator, denoted by $L,$ is called the (discrete) Laplacian. Here we adopt the sign convention such that $-L$ is a nonnegative operator. The associated $C_0$-semigroup is denoted by $P_t=e^{tL}:\ell^2(V,m)\to \ell^2(V,m).$
For locally finite graphs, the generator $L$ acts as
\begin{align*}
    L f(x)=\frac{1}{m(x)}\sum_{y\sim x}\mu_{xy}(f(y)-f(x)), \ \ \forall f\in C_0(V),
\end{align*} see \cite[Theorem~6~and~9]{KellerLenz12}.
Obviously, the measure $m$ plays an essential role in the definition of the Laplacian. Given the weight $\mu$ on $E,$ typical choices of $m$ of particular interest are: \begin{itemize}\item $m(x)=\sum_{y\sim x}\mu_{xy}$ for any $x\in V$ and the associated Laplacian is called the normalized Laplacian.
\item $m(x)=1$ for any $x\in V$ and the Laplacian is called combinatorial (or physical) Laplacian.
\end{itemize} Note that normalized Laplacians are bounded operators, so that these graphs are always stochastically complete, see Dodziuk \cite{Dodziuk06} or Keller and Lenz \cite{KellerLenz10}. Thus, the only interesting cases are combinatorial Laplacians, or more general unbounded Laplacians.

Following the strategy in \cite{BakryGentilLedoux}, to show stochastic completeness for the semigroups associated to unbounded Laplacians on graphs, it suffices to prove the gradient bounds as in
\eqref{e:gradient bounds}. For that purpose, we first introduce a completeness condition for infinite graphs:  A graph $G=(V,E,m,\mu)$ is called \emph{complete} if there exists a nondecreasing sequence of finitely supported functions $\{\eta_k\}_{k=1}^\infty$ such that \begin{equation}\label{d:complete}\lim_{k\to\infty}\eta_k=\mathds{1}\ \mathrm{and}\ \ \Gamma(\eta_k)\leq \frac{1}{k},\end{equation} where $\mathds{1}$ is the constant function $1$ on $V.$ Note that the measure $m$ plays a role in the definition of $\Gamma,$ see Definition~\ref{d:carre du}, so that it is essential to the completeness of a weighted graph. This condition was defined for Markov diffusion semigroups in \cite[Definition~3.3.9]{BakryGentilLedoux}; here we adapt it to graphs. As is well-known, this condition is equivalent to the geodesic completeness for Riemannian manifolds, see \cite{Strichartz83}. For the discrete setting, this condition is satisfied for a large class of graphs which possess intrinsic metrics, see Theorem \ref{thm:complete intrinsic}.

For gradient bounds \eqref{e:gradient bounds}, besides completeness we need curvature dimension conditions. For Markov diffusion semigroups, the curvature dimension conditions are defined via the $\Gamma$ operator and the iterated operator denoted by $\Gamma_2,$ see \cite[eq.~1.16.1]{BakryGentilLedoux}. This approach, using curvature dimension conditions to obtain gradient bounds, was initiated in Bakry and \'Emery \cite{BakryEmery85}. The curvature dimension condition on graphs, the non-diffusion case, was first introduced by Lin and Yau \cite{LinYau10} which serves as a combination of a lower bound of Ricci curvature and an upper bound of the dimension, see Definition \ref{d:curvature dimension} for an infinite dimensional version $\CCD (K,\infty).$ For bounded Laplacians on graphs, Bauer et al. \cite{Bauer13} introduced an involved curvature dimension condition, the so-called $\CDE(K,n)$ condition, to prove the Li-Yau gradient estimate for heat semigroups. Also restricted to bounded Laplacians, Lin and Liu \cite{LinLiu} proved the equivalence between the $\CCD(K,\infty)$ condition and the gradient bounds \eqref{e:gradient bounds} for heat semigroups, see Liu and Peyerimhoff \cite{LiuP14} for finite graphs. In this paper, under some mild assumptions, we prove the gradient bounds for unbounded Laplacians on graphs.


\begin{thm}[see Theorem~\ref{thm:main}]\label{thm:main theorem1}
  Let $G=(V,E,m,\mu)$ be a complete graph and $m$ be non-degenerate, i.e. $\inf_{x\in V}m(x)>0.$ Then the following are equivalent:
  \begin{enumerate}[(a)]
    \item $G$ satisfies $\CCD(K,\infty).$
    \item For any $f\in C_0(V),$
    $$\Gamma(P_t f)\leq e^{-2Kt}P_t(\Gamma (f)).$$
    \end{enumerate}
\end{thm}
Since it is not clear what volume growth is possible under the $\CCD(K,\infty)$ condition, our result cannot be derived from the criteria involving volume growth conditions. For unbounded Laplacians, the standard differential techniques for bounded Laplacians as in \cite{LiuP14,LinLiu} fail due to essential difficulties in the summability of solutions to heat equations. For instance, we don't know whether $\Gamma(P_t f)$ lies in the form domain (or, more strongly, in the domain of the generator), see Remark~\ref{r:difficulty}. In order to overcome these difficulties, we add a mild assumption on the measure $m,$ i.e. the non-degenerancy of the measure, and critically utilize techniques from partial differential equations, see Lemma~\ref{l:Caccippoli} for the Caccioppoli inequality and Theorem~\ref{thm:w12}. The assumption of the non-degenerancy of the measure $m$ is mild since it is automatically satisfied for any combinatorial Laplacian.

A direct consequence of the gradient bounds is the stochastic completeness for graphs satisfying the $\CCD(K,\infty)$ condition.
\begin{thm}\label{thm:main theorem2} Let $G=(V,E,\mu,m)$ be a complete graph satisfying the $\CCD(K,\infty)$ condition for some $K\in \R.$ Suppose that the measure $m$ is non-degenerate, then $G$ is stochastically complete.
 \end{thm}

The paper is organized as follows: In next section, we set up basic notations of weighted graphs. The $\Gamma$-calculus is introduced to define curvature dimension conditions. We define a new concept on the completeness of a graph and prove the completeness under the assumptions involving intrinsic metrics on graphs. In Section~\ref{s:Caccioppoli}, we adopt some PDE techniques to prove a (discrete) Caccioppoli inequality for Poisson's equations. In section~\ref{s:stochastic completeness}, we prove our main results: the
equivalence of curvature dimension conditions and the gradient bounds for heat semigroups on complete graphs, Theorem~\ref{thm:main theorem1}, and the stochastic completeness for graphs satisfying the curvature dimension condition, Theorem~\ref{thm:main theorem2}.

\section{Graphs}\label{s:graphs}
\subsection{Weighted graphs}
Let $(V,E)$ be a (finite or infinite) undirected graph with the set of vertices $V$ and the set of edges $E$ where $E$ is a symmetric subset of $V\times V.$ Two vertices $x,y$ are called neighbors if $(x,y)\in E,$ in this case denoted by $x\sim y.$ At a vertex $x,$ if $(x,x)\in E,$ we say there is a self-loop at $x.$ In this paper, we do allow self-loops for graphs. A graph $(V,E)$ is called connected if for any $x,y\in V$ there is a finite sequence of vertices, $\{x_i\}_{i=0}^n,$ such that $$x=x_0\sim x_1\sim \cdots\sim x_n=y.$$ In this paper, we only consider locally finite connected graphs.

We assign weights, $m$ and $\mu,$ on the set of vertices $V$ and edges $E$ respectively and refer to the quadruple $G=(V,E,m,\mu)$ as a \emph{weighted graph}: Here $\mu:E\to (0,\infty), E\ni (x,y)\mapsto \mu_{xy}$ is symmetric, i.e. $\mu_{xy}=\mu_{yx}$ for any $(x,y)\in E,$  and $m: V\to (0,\infty)$ is a measure on $V$ of full support. For convenience, we extend the function $\mu$ on $E$ to the total set $V\times V,$ $\mu:V\times V\to [0,\infty),$ such that $\mu_{xy}=0$ for any $x\not\sim y.$

For functions defined on $V,$ we denote by $\ell^p(V,m)$ or simply $\ell^p_m,$ the space of $\ell^p$ summable functions w.r.t. the measure $m$ and by $\|\cdot\|_{\ell^p_m}$ the $\ell^p$ norm of a function.
Given a weighted graph $(V,E,m,\mu)$, there is an associated Dirichlet form w.r.t. $\ell^2_m$ corresponding to the Neumann boundary condition, see \cite{HaeselerKellerLenzWojciechowski12},
\begin{eqnarray*}&&Q^{(N)}:D(Q^{(N)})\times D(Q^{(N)})\to \R \nonumber\\&&\ \ \   (f,g)\mapsto Q^{(N)}(f,g):=\frac12\sum_{x, y\in V}\mu_{xy}(f(y)-f(x))(g(y)-g(x)),\end{eqnarray*} where $D(Q^{(N)}):=\{f\in \ell^2_m|\  \sum_{x, y}\mu_{xy}(f(y)-f(x))^2<\infty\}.$ For simplicity, we write $Q^{(N)}(f):=\frac12\sum_{x, y}\mu_{xy}(f(y)-f(x))^2$ for any $f: V\to\R.$ Let $D(Q^{(D)})$ denote the completion of $C_0(V)$ under the norm $\|\cdot\|_Q$ defined by $$\|f\|_{Q}=\sqrt{\|f\|_{\ell^2_m}^2+Q^{(N)}(f)}, \ \ \forall f\in C_0(V).$$ Another Dirichlet form $Q^{(D)},$ defined as the restriction of $Q^{(N)}$ to $D(Q^{(D)}),$ corresponds to the Dirichlet boundary condition, see \eqref{e:dirichlet form} in Section~\ref{s:introduction}.

For the Dirichlet form $Q^{(N)},$ there is a unique self-adjoint
operator $L^{(N)}$ on $\ell^2_m$ with
$$D(Q^{(N)}) = \mathrm{Domain\ of\ definition\ of\ } (-L^{(N)})^{\frac12}$$
and
$$Q^{(N)}(f, g) = \left\langle(-L^{(N)})^{\frac12}f,(-L^{(N)})^{\frac12}g\right\rangle,\quad f,g\in D(Q^{(N)})$$ where $\langle\cdot,\cdot\rangle$ denotes the inner product in $\ell^2_m.$
The operator $L^{(N)}$ is the infinitesimal generator associated to the Dirichlet form $Q^{(N)},$ also called the (Neumann) Laplacian. The associated $C_0$-semigroup on $\ell^2_m$ is denoted by $P_t^{(N)}=e^{tL^{(N)}}.$ For the Dirichlet form $Q^{(D)}$, $L^{(D)}$ and $P_t^{(D)}$ are defined in the same way. In case that the Dirichlet forms corresponding to Neumann and Dirichlet boundary conditions coincide, i.e. $$Q^{(N)}=Q^{(D)},$$
we omit the superscripts and simply write $$Q=Q^{(N)}=Q^{(D)}, \quad L=L^{(N)}=L^{(D)}\quad \mathrm{etc}.$$

The following integration by parts formula is useful in further applications, see \cite[Corollary~1.3.1]{Fukushima}.
\begin{lemma}[Green's formula]\label{l:green formula}
  Let $(V,E,m,\mu)$ be a weighted graph. Then for any $f\in D(Q^{(N)})$ and $g\in D(L^{(N)})$,
  \begin{equation}\label{e:green formula}\sum_{x\in V} f(x)L^{(N)} g(x)m(x)=-Q^{(N)}(f,g).\end{equation} A similar consequence holds for the case of Dirichlet boundary condition.
\end{lemma}

For locally finite graphs, we define the \emph{formal Laplacian}, denoted by $\Delta,$ as
$$\Delta f(x)=\frac{1}{m(x)}\sum_{y\in X}\mu_{xy}(f(y)-f(x))\quad \forall\ f:V\to\R.$$ This formal Laplacian can be used to identify the generators defined before. A result of Keller and Lenz, \cite[Theorem~9]{KellerLenz12}, states that
\begin{equation}\label{e:generator action}L^{(D)}f=\Delta f,\quad \quad \forall\ f\in D(L^{(D)}),\end{equation} and a similar result holds for Neumann condition, see \cite{HaeselerKellerLenzWojciechowski12}. Note that$$\Delta f\in C_0(V), \ \ \forall\ f\in C_0(V).$$ Different choices for the measure $m$ induce different Laplacians. The typical choices are normalized Laplacians and combinatorial Laplacians, see Section~\ref{s:introduction}.



The measure $m$ on $V$ is called non-degenerate if \begin{equation}\label{d:nondegenerate}\delta:=\inf_{x\in V}m(x)>0.\end{equation} The nondegerancy of the measure $m$ yields a very useful fact for $\ell^p(V,m)$ spaces.
\begin{prop}\label{p:nondegenerate} Let $m$ be a non-degenerate measure on $V$ as in \eqref{d:nondegenerate}. Then for any $f\in \ell^p(V,m),$ $p\in [1,\infty),$
$$|f(x)|\leq \delta^{-\frac1p}\|f\|_{\ell^p_m}\ \ \ \forall x\in V.$$ Moreover, for any $1\leq p<q\leq \infty,$ $\ell^p(V,m)\hookrightarrow\ell^q(V,m).$
\end{prop}
\begin{proof}
  The first assertion follows from $|f(x)|^p\delta\leq |f(x)|^pm(x)\leq \|f\|_{\ell^p_m}^p.$ The second one is a consequence of the interpolation theorem.
\end{proof}

Under assumptions of non-degeneracy of the measure $m$ and local finiteness of the graph, the Dirichlet forms corresponding to Neumann and Dirichlet boundary conditions coincide, i.e.
$$Q^{(N)}=Q^{(D)},$$ see \cite[Theorem~6]{KellerLenz12} and \cite[Corollary~5.3]{HaeselerKellerLenzWojciechowski12}
and the domains of generators are characterized as
$$D(L^{(N)})=D(L^{(N)})=\{f\in\ell^2_m|\ \Delta f\in \ell^2_m\}.$$

\subsection{Gamma calculus}
We introduce the $\Gamma$-calculus and curvature dimension conditions on graphs following \cite{LinYau10,Bauer13}.

 First we define two natural bilinear forms associated to the Laplacian. Given $f: V\to\R$ and $x,y\in V,$ we denote by $\nabla_{xy}f:= f(y)-f(x)$ the difference of the function $f$ on the vertices $x$ and $y.$
\begin{defi}\label{d:carre du}
  The gradient form $\Gamma,$ called the ``carr\'e du champ" operator, is defined by
  \begin{eqnarray*}\Gamma(f,g)(x)&=&\frac12(\Delta(fg)-f\Delta g-g\Delta f)(x)\\&=&\frac{1}{2m(x)}
  \sum_{y}\mu_{xy}\nabla_{xy}f\nabla_{xy}g.\end{eqnarray*} For simplicity, we write $\Gamma(f):=\Gamma(f,f).$ Moreover, the iterated gradient form, denoted by $\Gamma_2$, is defined as $$\Gamma_2(f,g)=\frac12(\Delta\Gamma(f,g)-\Gamma(f,\Delta g)-\Gamma(g,\Delta f)).$$ We write $\Gamma_2(f):=\Gamma_2(f,f)=\frac{1}{2}\Delta \Gamma(f)-\Gamma(f,\Delta f).$
\end{defi}
The Cauchy-Schwarz inequality implies that \begin{equation}\label{eq:gamma}\Gamma(f,g)\leq \sqrt{\Gamma(f)\Gamma(g)}\leq \frac12 (\Gamma(f)+\Gamma(g)).\end{equation} In addition, one can easily see that $Q^{(N)}(f)=\|\Gamma(f)\|_{\ell^1_m}.$

Now we can introduce curvature dimension conditions on graphs.
\begin{defi}\label{d:curvature dimension}
  We say a graph $(V,E,m,\mu)$ satisfies the $\CCD(K,\infty)$ condition, $K\in \R$, if for any $x\in V,$
  $$\Gamma_2(f)(x)\geq K\Gamma (f)(x).$$
\end{defi}
One can also define a finite dimensional version, $\CCD(K,n)$ condition (see \cite{LinYau10}), which is stronger than $\CCD(K,\infty).$ An involved curvature dimension condition, called $\CDE(K,n),$ was introduced in \cite{Bauer13}. In this paper, we only consider $\CCD(K,\infty)$ conditions.

\subsection{Completeness of graphs}
Yau \cite{Yau78} first proved that complete Riemannian manifolds with Ricci curvature uniformly bounded from below are stochastically complete.
Bakry \cite{Bakry86} proved the stochastic completeness for weighted Riemannian manifolds satisfying $\CCD(K,\infty)$ condition for weighted Laplacians, see also Li \cite{Lixiangdong05}. The completeness of Riemannian manifolds plays an important role in these problems.

For a graph $(V,E,m,\mu),$ we define the completeness of a graph as in \eqref{d:complete}, see Section~\ref{s:introduction}. The following lemma shows the importance of the completeness of a graph. Note that we don't need the non-degenerancy of the measure $m$ here.
\begin{lemma}\label{l:Qdensity} Let $(V,E,m,\mu)$ be a complete graph.
  For any $f\in \ell^2_m$ such that $Q^{(N)}(f)<\infty$ we have $$\|f\eta_k-f\|_Q\to 0,\ \ \ k\to\infty.$$ That is, $C_0(V)$ is a dense subset of the Hilbert space $(D(Q^{(N)}),\|\cdot\|_Q)$ and $Q^{(N)}=Q^{(D)}.$
\end{lemma}

\begin{proof}
  It is easy to see that $f_k:=f\eta_k\to f$ in $\ell^2_m.$ So it suffices to show that $Q^{(N)}(f_k-f)\to 0, $ $k\to \infty.$
  \begin{eqnarray*}
    Q^{(N)}(f_k-f)&=&\frac12\sum_{x,y}\mu_{xy}|\nabla_{xy}f(\eta_k-1)|^2\\
    &=&\frac{1}{2}\sum_{x,y}\mu_{xy}|\nabla_{xy}f \cdot(\eta_k(y)-1)+f(x)\nabla_{xy}\eta_k|^2\\
    &\leq&\sum_{x,y}\mu_{xy}(|\nabla_{xy}f|^2 |\eta_k(y)-1|^2+f^2(x)|\nabla_{xy}\eta_k|^2)\\
    &=&I_k+II_k.
  \end{eqnarray*} By the dominated convergence theorem, $I_k\to 0$ as $k\to\infty.$ For the second term,
  $$II_k\leq\frac{2}{k^2}\sum_{x}f^2(x)m(x)\to 0,\ \ \ k\to\infty.$$ This proves the lemma.
\end{proof}

Hence for a complete graph, $Q^{(N)}=Q^{(D)}.$ In the rest of the paper, given a complete graph we simply write $Q=Q^{(N)}=Q^{(D)},$ and by \eqref{e:generator action}
$$L=\Delta,\quad \quad\quad  \mathrm{on}\quad \ D(Q).$$

\subsection{Intrinsic metrics}\label{s:intrinsic}
The Laplacian associated with the graph $(V,E,m,\mu)$ is a bounded operator from $\ell^2(V,m)$ to $\ell^2(V,m)$ if and only if $$\sup_{x\in V} \frac{1}{m(x)}\sum_{y\sim x}\mu_{xy}<\infty.$$
In order to deal with unbounded Laplacians, we need the following intrinsic metrics on graphs introduced in \cite{FrankLenzWingert12}.

A pseudo metric $\rho$ is a symmetric function, $\rho:V\times V\to[0,\infty),$ with zero diagonal which satisfies the triangle inequality.

\begin{defi}[Intrinsic metric]\label{d:intrinsic} A pseudo metric $\rho$ on $V$ is called \emph{intrinsic} if
\begin{align*}
\sum_{y\in V}\mu_{xy}\rho^{2}(x,y)\leq m(x),\qquad x\in V.
\end{align*}
\end{defi}

In various situations the natural graph distance, called the combinatorial distance, proves to be insufficient for the investigations of unbounded Laplacians, see \cite{Wojciechowski1,Wojciechowski2,KellerLenzWojciechowski}. For this reason the concept of intrinsic metrics received quite some attention as a candidate to overcome these problems. Indeed, intrinsic metrics already have been applied successfully to various problems on graphs \cite{BauerHuaKeller,BauerKellerWojciechowski,FOLZ,Folz14,GrigoryanHuangMasamune12,HuangKellerMasamuneWojciechowski,HuaKeller14}.

Fix a base point  $o\in V$ and
denote the distance balls by
\begin{align*}
    B_{r}(o)=\{x\in V\mid\rho(x,o)\leq r\},\qquad r\ge0.
\end{align*}
The choice of the base point $o$ will be irrelevant to our results later. We say $B_{r}(o)$ is finite, if it is of finite cardinality, i.e. $\sharp B_{r}(o)<\infty.$

\begin{thm}\label{thm:complete intrinsic}
  Let $G=(V,E,m,\mu)$ be a graph and $\rho$ be an intrinsic metric on $G.$ Suppose that each ball $B_r(o),$ $r>0$, is finite, then
  $G$ is a complete graph.
\end{thm}
\begin{proof}
  For any $0<r<R,$ we denote by $\eta_{r,R}$ the cut-off function on $B_R(o)\setminus B_r(o)$ defined as
\begin{align*}
    \eta_{r,R}(\cdot)=\min\left\{\max\left\{\frac{R-\rho(\cdot,o)}{R-r},0\right\},1\right\}.
\end{align*} Set $\eta_k:=\eta_{k,2k}.$ Then $\{\eta_k\}$ is a nondecreasing sequence of finitely supported functions which converges to the constant function $\mathds{1}$ pointwise. Moreover,
\begin{eqnarray*}
  \Gamma(\eta_k)(x)&=&\frac{1}{2m(x)}\sum_{y\in V}\mu_{xy}|\nabla_{xy}\eta_k|^2\\
  &\leq& \frac{1}{2m(x)k^2}\sum_{y\in V}\mu_{xy}\rho^2(x,y)\\
  &\leq&\frac{1}{2k^2}<\frac{1}{k},
\end{eqnarray*} where we used the definition of the intrinsic metric $\rho.$ This proves the theorem.
\end{proof}

For any weighted graph $(V,E,m,\mu),$ intrinsic metrics always exist. There is a natural intrinsic metric introduced by Huang \cite[Lemma~1.6.4]{Huang11}. Define the \emph{weighted vertex degree} $\mathrm{Deg}:V\to[0,\infty)$ by
\begin{align*}
    \mathrm{Deg}(x)=\frac{1}{m(x)}\sum_{y\in V}\mu_{xy},\qquad x\in V.
\end{align*}

\begin{example}\label{ex:intrinsic}
For any  given weighted graph there is an intrinsic path metric defined by
\begin{align*}
    \de(x,y)=\inf_{x=x_{0}\sim\ldots\sim x_{n}=y}\sum_{i=0}^{n-1} (\mathrm{Deg}(x_{i})\vee\mathrm{Deg}(x_{i+1}))^{-\frac{1}{2}},\quad\quad x,y\in V,
\end{align*} where the infimum is taken over all finite paths connecting $x$ and $y.$

\end{example}
For the completeness of the graph, it suffices to find an intrinsic metric satisfying the conditions in Theorem~\ref{thm:complete intrinsic}. For instance, one can check whether each ball of finite radius under the metric $\delta$ is finite.

\section{Semigroups and Caccioppoli inequality}\label{s:Caccioppoli}
\subsection{Semigroups on graphs}
In this section, we study the properties of heat semigroups on graphs, which will be used later.

We denote by $P_t^{(D)}=e^{tL^{(D)}}$ the $C_0$-semigroup associated to the Dirichlet form $Q^{(D)}$ on $\ell^2_m.$ It extrapolates to $C_0$-semigroups on $\ell^p_m$ for all $p\in [1,\infty],$ for simplicity still denoted by $P_t^{(D)},$ see \cite{KellerLenz12}.
\begin{defi}\label{d:stochastic complete graph}
  A weighted graph $(V,E,m,\mu)$ is called \emph{stochastically complete} if
  $$P_t^{(D)}\mathds{1}=\mathds{1},\quad \forall\ t>0,$$ where $\mathds{1}$ is the constant function $1$ on $V.$
\end{defi}
The next proposition is a consequence of standard Dirichlet form theory, see \cite{Fukushima} and \cite{KellerLenz12}.
\begin{prop}\label{p:basic pt}
  For any $f\in \ell^p_m,$ $p\in[1,\infty],$ we have $P_t^{(D)} f\in \ell^p_m$ and
  $$\|P_t^{(D)} f\|_{\ell^p_m}\leq \|f\|_{\ell^p_m},\ \ \ \forall t\geq 0.$$ Moreover, $P_t^{(D)} f\in D(L^{(D)})$ for any $f\in \ell^2_m.$
\end{prop}
The next property follows from the spectral theorem.
\begin{prop}\label{p:exchange}
  For any $f\in D(L^{(D)}),$ $$L^{(D)} P_t^{(D)} f=P_t^{(D)}L^{(D)} f.$$
\end{prop}

\subsection{Caccioppoli inequality}
For elliptic partial differential equations on Riemannian manifolds, the Caccioppoli inequality is well-known and yields the $L^p$ Liouville theorem for harmonic functions for $p\in (1,\infty),$ see Yau \cite{Yau76}.

By adapting PDE techniques on manifolds to graphs, we obtain the Caccioppoli inequality for subsolutions to Poisson's equations.
\begin{lemma}\label{l:Caccippoli}
  Let $(V,E,m,\mu)$ be a weighted graph and $g,h:V\to\R$ satisfy the following $$\Delta g\geq h.$$ Then for any $\eta\in C_0(V),$
  \begin{equation}\label{e:caccioppoli}
    \|\Gamma(g) \eta^2\|_{\ell^1_m}\leq C(\|\Gamma(\eta) g^2\|_{\ell^1_m}+\|gh \eta^2\|_{\ell^1_m}).
  \end{equation}
\end{lemma}
\begin{proof}
Multiplying $\eta^2 g$ to both sides of the inequality, $\Delta g\geq h,$ and summing over $x\in V$ w.r.t. the measure $m,$ we get
\begin{eqnarray*}
  &&\sum_x \eta^2gh (x) m(x)\leq \sum_x \eta^2 g \Delta g(x) m(x)\\
  &=&-\frac12\sum_{x,y}\nabla_{xy}g\nabla_{xy}(\eta^2 g)\mu_{xy}\\
  &=&-\frac12\sum_{x,y}\nabla_{xy}g(\nabla_{xy}g \eta^2(x)+g(y)\nabla_{xy}(\eta^2))\mu_{xy}\\
  &=&-\frac12\sum_{x,y}|\nabla_{xy}g|^2 \eta^2(x)\mu_{xy}-\frac12\sum_{x,y}\nabla_{xy}gg(y) (|\nabla_{xy}\eta|^2+2\eta(x)\nabla_{xy}\eta)\mu_{xy},
\end{eqnarray*} where we used Green's formula, see e.g. Lemma~\ref{l:green formula}, in the second line since $\eta\in C_0(V).$ For the second term in the last line, by symmetry one has
$$-\frac12\sum_{x,y}\nabla_{xy}gg(y) |\nabla_{xy}\eta|^2\mu_{xy}=-\frac{1}{4}\sum_{x,y}|\nabla_{xy}g|^2 |\nabla_{xy}\eta|^2\mu_{xy}\leq 0.$$ Hence, by this observation, the previous estimate leads to
\begin{eqnarray*}
  &&\frac12\sum_{x,y}|\nabla_{xy}g|^2 \eta^2(x)\mu_{xy}\\
  &\leq&-\sum_{x,y}\nabla_{xy}gg(y) \eta(x)\nabla_{xy}\eta\mu_{xy}-\sum_x \eta^2gh (x) m(x)\\
  &\leq& \frac14\sum_{x,y}|\nabla_{xy}g|^2\eta^2(x)\mu_{xy}+\sum_{x,y}|\nabla_{xy} \eta|^2g^2(y)\mu_{xy}-\sum_x \eta^2gh (x) m(x),
\end{eqnarray*} where we used basic inequality $ab\leq \frac14a^2+b^2$ for $a,b\in\R.$ The lemma follows from cancelling the first term in the last line with the left hand side of the system of inequalities.

\end{proof}

Using this Caccippoli inequality, we get a uniform upper bound of the Dirichlet energy of $P_tf$ for $t>0$ and $f\in C_0(V).$
\begin{lemma}\label{lem:energy estimate}
  Let $(V,E,m,\mu)$ be a complete graph. Then for any $f\in C_0(V)$ and $t\in [0,\infty),$
  $$Q(P_tf)=\|\Gamma(P_t f)\|_{\ell^1_m}\leq C \|f\|_{\ell^2_m}\|\Delta f\|_{\ell^2_m},$$ where $C$ is a uniform constant.
\end{lemma}
\begin{proof}
  For $f\in C_0(V),$ the local finiteness of the graph implies that $\Delta f\in C_0(V).$
  By the completeness of the graph, let $\eta_k\in C_0(V)$ satisfy \eqref{d:complete}. Since $P_tf$ satisfies the equation $\frac{d}{dt} P_tf=\Delta P_t f$ for any $t>0,$ applying the Caccippoli inequality in Lemma \ref{l:Caccippoli} with $g=P_t f,$ $h=\frac{d}{dt} P_t f$ and $\eta=\eta_k,$ we have
      \begin{eqnarray*}\|\Gamma(P_t f) \eta_k^2\|_{\ell^1_m}&\leq& C(\|\Gamma(\eta_k) |P_t f|^2\|_{\ell^1_m}+\|P_t f \cdot\frac{d}{dt} P_t f\cdot \eta_k^2\|_{\ell^1_m})\\
      &\leq& C\left(\frac1k\| P_t f\|_{\ell^2_m}^{2}+\|P_t f\|_{\ell^2_m}\|\frac{d}{dt} P_t f\|_{\ell^2_m}\right).\end{eqnarray*}
By Proposition \ref{p:basic pt}, $$\|P_t f\|_{\ell^2_m}\leq \|f\|_{\ell^2_m}$$ and by Proposition \ref{p:exchange} and the equation \eqref{e:generator action},
$$\|\frac{d}{dt} P_t f\|_{\ell^2_m}=\|\Delta P_t f\|_{\ell^2_m}=\|P_t\Delta f\|_{\ell^2_m}\leq \|\Delta f\|_{\ell^2_m}. $$ Hence $$\|\Gamma(P_t f) \eta_k^2\|_{\ell^1_m}\leq C\left(\frac1k\|f\|_{\ell^2_m}^2+\|f\|_{\ell^2_m}\|\Delta f\|_{\ell^2_m}\right).$$ By passing to the limit, $k\to \infty,$ the monotone convergence theorem yields the lemma.
\end{proof}

The following result is an improved estimate of the previous lemma which will be useful in further applications.
\begin{lemma}\label{lem:max time estimate}
  Let $(V,E,m,\mu)$ be a complete graph. Then for any $f\in C_0(V)$ and $T>0,$ we have $\max_{[0,T]}\Gamma(P_t f)\in\ell^1_m$ and
  \begin{equation}\label{eq:max time}\left\|\max_{[0,T]}\Gamma(P_t f)\right\|_{\ell^1_m}\leq C_1(T,f), \end{equation} where $C_1(T,f)$ is a constant depending on $T$ and $f.$ Moveover, $$\max_{[0,T]}|\Gamma(P_t f,\frac{d}{dt} P_t f)|\in\ell^1_m\quad\mathrm{and}$$
  \begin{equation}\label{eq:mixed estimate}
    \left\|\max_{[0,T]}|\Gamma(P_t f,\frac{d}{dt} P_t f)|\right\|_{\ell^1_m}=\left\|\max_{[0,T]}|\Gamma(P_t f,\Delta P_t f)|\right\|_{\ell^1_m}\leq C_2(T,f).
  \end{equation}
\end{lemma}
\begin{proof}
The local finiteness yields that $\Delta f\in C_0(V)$ and $\Delta^2 f\in C_0(V)$ for $f\in C_0(V).$

For the first assertion, the Newton-Leibniz formula yields
  \begin{eqnarray*}\Gamma(P_tf)&=&\Gamma(f)+\int_0^T\frac{d}{ds}\Gamma(P_sf)ds\\
  &=&\Gamma(f)+2\int_0^T\Gamma(P_sf,\frac{d}{ds}P_sf)ds\\
  &=&\Gamma(f)+2\int_0^T\Gamma(P_sf,\Delta P_sf)ds\\
  &=&\Gamma(f)+2\int_0^T\Gamma(P_sf,P_s(\Delta f))ds,\end{eqnarray*} where the last equality follows from Proposition \ref{p:exchange}. Hence by the equation \eqref{eq:gamma} and Lemma \ref{lem:energy estimate}
  \begin{eqnarray*}\left\|\max_{[0,T]}\Gamma(P_tf)\right\|_{\ell_m^1}&\leq& \|\Gamma(f)\|_{\ell^1_m}+2\left\|\int_0^T|\Gamma(P_sf,P_s(\Delta f))|ds\right\|_{\ell^1_m}\\
  &\leq&\|\Gamma(f)\|_{\ell^1_m}+\int_0^T(\|\Gamma(P_sf)\|_{\ell^1_m}+\|\Gamma(P_s(\Delta f))\|_{\ell^1_m})ds\\
  &\leq&\|\Gamma(f)\|_{\ell^1_m}+CT\|\Delta f\|_{\ell^2_m}(\|f\|_{\ell^2_m}+\|\Delta^2 f\|_{\ell^2_m})=:C_1(T,f). \end{eqnarray*}

  The second assertion is a direct consequence of the first one. By $\Delta f\in C_0(V)$ and \eqref{eq:gamma},
  \begin{eqnarray*}
    \left\|\max_{[0,T]}|\Gamma(P_t f,\frac{d}{dt} P_t f)|\right\|_{\ell^1_m}&=&\left\|\max_{[0,T]}|\Gamma(P_t f,\Delta P_t f)|\right\|_{\ell^1_m}=\left\|\max_{[0,T]}|\Gamma(P_t f,P_t \Delta f)|\right\|_{\ell^1_m}\\
    &\leq&\frac12\left\|\max_{[0,T]}\Gamma(P_t f)\right\|_{\ell^1_m}+\frac12\left\|\max_{[0,T]}\Gamma(P_t \Delta f)\right\|_{\ell^1_m}\\
    &\leq& \frac12(C_1(T,f)+C_1(T,\Delta f))=:C_2(T,f).
  \end{eqnarray*}

   This proves the lemma.
\end{proof}

Now we can show that the Dirichlet energy, $t\mapsto Q(P_t f)$, decays in time for the semigroup $P_t$ on complete graphs.
\begin{prop}\label{p:Dirichlet energy}Let $(V,E,m,\mu)$ be a complete graph. Then for any $f\in C_0(V),$
  $$Q(P_t f)\leq Q(f),\ \ \ \forall t\geq 0.$$
  Moreover, for any $f\in D(Q),$
  $$Q(P_t f)\leq Q(f),\ \ \ \forall t\geq 0.$$
\end{prop}
\begin{proof}
For the first assertion, taking the formal derivative of time in $Q(P_tf)$ for $t>0,$ we get
\begin{equation}\label{eq:eq1}
  \frac{d}{dt}Q(P_t f)= 2\sum_{x\in V}\Gamma(P_t f,\frac{d}{dt} P_t f)(x)m(x).
\end{equation}
Given a fixed $T>t,$ note that for any $t\in [0,T],$
 \begin{eqnarray*}
   |\Gamma(P_t f,\frac{d}{dt} P_t f)(x)|&\leq&\max_{t\in [0,T]}|\Gamma(P_t f,\frac{d}{dt} P_t f)(x)|=: g(x)\in\ell^1_m
    \end{eqnarray*} which follows from \eqref{eq:mixed estimate} in Lemma \ref{lem:max time estimate}. Hence the absolute value of the summand on the right hand side of \eqref{eq:eq1} is uniformly (for $t\in [0,T]$) bounded above by a summable function $g.$ The differentiability theorem yields that $Q(P_t f)$ is differentiable in time and whose the derivative is given by \eqref{eq:eq1}.

 Since $P_t f\in D(L)$ and $\Delta P_t f=P_t \Delta f\in D(Q),$ Green's formula in Lemma \ref{l:green formula} yields \begin{eqnarray*}
  \frac{d}{dt}Q(P_t f)&=& 2\sum_{x\in V}\Gamma(P_t f, \Delta P_t f)(x)m(x)\\
  &=&-2 \sum_{x\in V} |\Delta P_t f (x)|^2m(x)\leq 0.
\end{eqnarray*} This proves the first assertion.

For the second assertion, set $f_k:=f\eta_k$ for $f\in D(Q).$ It follows from the previous result that
$$Q(P_tf_k)\leq Q(f_k).$$ By Lemma \ref{l:Qdensity}, $f_k\to f$ in the norm $\|\cdot\|_{Q}.$ The monotone convergence theorem yields that $$P_tf_k\to P_tf$$ pointwise. By Fatou's lemma, $$Q(P_tf)\leq \liminf_{k\to\infty} Q(P_t f_k)\leq \liminf_{k\to\infty} Q(f_k)=Q(f).$$ This proves the theorem.
\end{proof}


\section{Stochastic completeness}\label{s:stochastic completeness}
\subsection{Gradient bounds and curvature dimension conditions}
The curvature dimension condition implies gradient bounds, see \cite{BakryGentilLedoux} for the case of Markov diffusion semigroups. In fact, they are equivalent on locally finite graphs under some mild assumptions.
\begin{thm}\label{thm:main}
  Let $G=(V,E,m,\mu)$ be a complete graph with a non-degenerate measure $m$, i.e. $\inf_{x\in V}m(x)>0.$ Then the following are equivalent:
  \begin{enumerate}[(a)]
    \item $G$ satisfies $\CCD(K,\infty).$
    \item For any $f\in C_0(V),$
    $$\Gamma(P_t f)\leq e^{-2Kt}P_t(\Gamma (f)).$$
    \item For any $f\in D(Q),$
    $$\Gamma(P_t f)\leq e^{-2Kt}P_t(\Gamma (f)).$$
  \end{enumerate}
\end{thm}

\begin{rem}\label{r:difficulty}
For the case of finite graphs or bounded Laplacians, this result has been proven by \cite{LiuP14,LinLiu}. To illustrate their proof strategy, we consider a finite graph $(V,E,m,\mu)$ satisfying the $\CCD(0,\infty)$ condition.

$(a)\Rightarrow(b):$ For any $f:V\to\R,$ set $\Lambda(s)=P_s(\Gamma(P_{t-s}f)).$
  Then \begin{eqnarray*}
    \Lambda'(s)&=&\Delta P_s(\Gamma(P_{t-s}f))-2P_s(\Gamma(P_{t-s}f,\Delta P_{t-s}f))\\
    &=&P_s(\Delta\Gamma(P_{t-s}f)-2\Gamma(P_{t-s}f,\Delta P_{t-s}f))\geq 0,
  \end{eqnarray*} where the last inequality follows from the $\CCD(0,\infty)$ condition. However, for the case of infinite graphs,  $\Delta P_s(\Gamma(P_{t-s}f)))=P_s\Delta(\Gamma(P_{t-s}f)))$ may not hold since in general we don't know whether $\Gamma(P_{t-s}f)\in D(L).$

  In addition, a strong version of gradient bounds has been proved using the following stronger curvature condition,
  see \cite[equation 3.2.4]{BakryGentilLedoux}
  \begin{equation}\label{e:strong condition}\Gamma(\Gamma(g))\leq 4\Gamma(g)[\Gamma_2(g)-K\Gamma(g)], \ \ \forall g\in C_0(V).\end{equation} However, this stronger curvature condition can never be fulfilled for graphs. In fact, the inequality \eqref{e:strong condition} fails e.g. for $g=\delta_x.$
\end{rem}

\subsection{Curvature dimension conditions and the properties of heat semigroups.}

In order to prove the gradient estimate under the $\CCD(K,\infty)$ condition, we need some lemmata. For graphs satisfying the $\CCD(K,\infty)$ condition, the following lemma states that $\Gamma(P_t f)$ is a subsolution to the heat equation, a standard definition in the theory of PDEs.

\begin{lemma}\label{l:heat subsolution}
  Let $(V,E,m,\mu)$ be a complete graph satisfying the $\CCD(K,\infty)$ condition. Then for any $f\in C_0(V)$
  $$\frac{d}{dt}\Gamma(P_t f)\leq \Delta \Gamma(P_t f)-2K\Gamma(P_t f).$$
\end{lemma}
\begin{proof}
 This follows from direct calculation by means of the $\CCD(K,\infty)$ condition and local finiteness of the graph.
\end{proof}

\begin{lemma}\label{l:l1 time derivative} Let $(V,E,m,\mu)$ be a complete graph. Then
  for any $f\in C_0(V)$ and $t\geq 0,$ $$\left\|\frac{d}{dt}\Gamma(P_t f)\right\|_{\ell^1_m}\leq 2\sqrt{Q(f)Q(\Delta f)}.$$
\end{lemma}
\begin{proof}
  This follows by the computation,
  \begin{eqnarray*}
    &&\left\|\frac{d}{dt}\Gamma(P_t f)\right\|_{\ell^1_m}=2\sum_{x}\left|\Gamma(P_t f, \frac{d}{dt}P_t f)(x)\right|m(x)\\
    &=&2\sum_{x}|\Gamma(P_t f, \Delta P_t f)(x)|m(x)=2\sum_{x}|\Gamma(P_t f, P_t \Delta f )(x)|m(x)\\
    &\leq&2 \sqrt{\sum_{x}\Gamma(P_t f)m(x)\sum_{x}\Gamma(P_t \Delta f)m(x)}\\
    &\leq &2 \sqrt{\sum_{x}\Gamma(f)m(x)\sum_{x}\Gamma(\Delta f)m(x)}<\infty,
  \end{eqnarray*} where we used Proposition \ref{p:exchange} for $f\in C_0(V)$ in the third equality and Proposition \ref{p:Dirichlet energy} for $f,\Delta f\in C_0(V)$ in the last one.
\end{proof}

For complete graphs satisfying the $\CCD(K,\infty)$ condition, we have higher summability of the solutions to heat equations.
\begin{thm}\label{thm:w12}
  Let $G=(V,E,m,\mu)$ be a complete graph with a non-degenerate measure $m$. If $G$ satisfies the $\CCD(K,\infty)$ condition, then for any $f\in C_0(V)$ and $t\geq 0,$
  $$\Gamma (P_tf)\in D(Q).$$
\end{thm}
\begin{proof}
   From the proof of Proposition \ref{p:Dirichlet energy}, $\Gamma(P_t f)\in \ell^1(V,m).$ Hence by the nondenegerancy of $m,$ $\Gamma(P_t f)\in \ell^2(V,m).$ It suffices to prove that $Q(\Gamma(P_t f))<\infty.$

  Let $\{\eta_k\}$ be the sequence in \eqref{d:complete} by the completeness of the graph.  Note that Lemma \ref{l:heat subsolution} implies that $\Gamma(P_t f)$ is a subsolution to the heat equation.
  Applying the Caccioppoli inequality \eqref{e:caccioppoli} with $g=\Gamma(P_t f),$ $h=\frac{d}{dt} g+2K g$ and $\eta=\eta_k,$
  we get \begin{eqnarray*}
      \|\Gamma(g) \eta_k^2\|_{\ell^1_m}&\leq& C(\|\Gamma(\eta_k) g^2\|_{\ell^1_m}+\|g(\frac{d}{dt} g+2K g) \eta_k^2\|_{\ell^1_m})\\
      &\leq & C\left(\frac{1}{k}\|g\|_{\ell^2_m}^2+\|g\frac{d}{dt} g\|_{\ell^1_m}+2|K|\cdot\|g\|_{\ell^2_m}^2\right)\\
      &\leq & C(K)(\|g\|_{\ell^2_m}^2+\|g\frac{d}{dt} g\|_{\ell^1_m})\\
       &=&I+II,
  \end{eqnarray*} where the constant $C(K)$ only depends on $K.$
  By the assumption that $m$ is non-degenerate, Propositions \ref{p:nondegenerate} and
  \ref{p:Dirichlet energy} yield that
  $$I\leq C\|\Gamma(P_tf)\|_{\ell^1_m}^2\leq C\|\Gamma(f)\|_{\ell^1_m}^2<\infty.$$

  For the other term, noting that $\|g\|_{\ell^\infty}\leq C\|g\|_{\ell^1_m},$ by Lemma \ref{l:l1 time derivative}, we have
  \begin{eqnarray*}
    II&\leq& C\|g\|_{\ell^\infty}\|\frac{d}{dt}g\|_{\ell^1_m}\\
    &\leq& C \|g\|_{\ell^1_m}\|\frac{d}{dt}g\|_{\ell^1_m}<\infty.
  \end{eqnarray*} Thus, $\|\Gamma(g) \eta_k^2\|_{\ell^1_m}\leq C<\infty$ where the right hand side is independent of $k.$ By passing to the limit, $k\to \infty,$ Fatou's lemma yields that
  $$\|\Gamma(\Gamma(P_t f))\|_{\ell^1_m}\leq \liminf_{k\to\infty}\|\Gamma(\Gamma(P_t f)) \eta_k^2\|_{\ell^1_m}\leq C.$$ This proves the theorem.

\end{proof}

\subsection{The proofs of main theorems}
\begin{thm}\label{thm:monotone}
  Let $(V,E,m,\mu)$ be a complete graph with a non-degenerate measure $m$ and satisfying the $\CCD(K,\infty)$ condition. For any $f\in C_0(V),$ $0\leq \zeta\in C_0(V)$ and $t>0,$ the following function
  $$s\mapsto G(s):=\sum_{x\in V}\Gamma(P_{t-s}f)(x)P_s \zeta(x)m(x)$$ satisfies
  $$G'(s)\geq 2K G(s),\quad\quad\quad 0<s<t.$$
\end{thm}
\begin{proof}
  First, we show that $G(s)$ is differentiable in $s\in(0,t)$. Without loss of generality, we assume that $\epsilon<s<t-\epsilon$ for some $\epsilon>0.$ Taking the formal derivative of $G(s)$ in $s,$ we get
  \begin{equation}\label{e:Gs1}
    -2\sum_{x}\Gamma(P_{t-s}f,\Delta P_{t-s}f)(x)P_s\zeta(x)m(x)+\sum_{x}\Gamma(P_{t-s}f)(x)\Delta (P_{s}\zeta)(x)m(x)
  \end{equation} This formal derivative is, in fact, the derivative of $G(s)$ if one can show that the absolute values of summands are uniformly (in $s$) controlled by summable functions. For the first term in \eqref{e:Gs1}, note that $\|P_s \zeta\|_{\ell^{\infty}}\leq \|\zeta\|_{\ell^{\infty}}<\infty.$ Then the equation \eqref{eq:mixed estimate} in Lemma \ref{lem:max time estimate} yields that for any $s\in (\epsilon, t-\epsilon)$
  \begin{eqnarray*}
    &&2|\Gamma(P_{t-s}f,\Delta P_{t-s}f)(x)|P_s\zeta(x)\leq\sup_{s\in (\epsilon,t-\epsilon)}2|\Gamma(P_{t-s}f,\Delta P_{t-s}f)(x)|P_s\zeta(x)\\&\leq& 2\|\zeta\|_{\ell^\infty}\sup_{s\in (\epsilon,t-\epsilon)} |\Gamma(P_{s}f,\Delta P_{s}f)(x)|=:g(x)\in\ell^1_m.
      \end{eqnarray*}
  For the second term in \eqref{e:Gs1}, the equation \eqref{eq:max time} in Lemma \ref{lem:max time estimate} implies that for any $s\in (\epsilon, t-\epsilon)$
  \begin{eqnarray*}
    \Gamma(P_{t-s}f)(x)|\Delta (P_{s}\zeta)(x)|&\leq&\sup_{s\in (\epsilon,t-\epsilon)}\Gamma(P_{t-s}f)(x)|\Delta (P_{s}\zeta)(x)|\\
    &=&\sup_{s\in (\epsilon,t-\epsilon)}\Gamma(P_{t-s}f)(x)|P_{s}\Delta \zeta(x)|\\
    &\leq& \|\Delta \zeta\|_{\ell^\infty}\sup_{s\in (\epsilon,t-\epsilon)}\Gamma(P_{s}f)(x)=:h(x)\in\ell^1_m.
  \end{eqnarray*} Since $g+h\in \ell^1_m$ which is independent of $s\in(\epsilon,t-\epsilon),$ the differentiability theorem yields that $G(s)$ is differentiable and its derivative equals to \eqref{e:Gs1}. Note that Theorem \ref{thm:w12} and Proposition \ref{p:basic pt} yield $\Gamma(P_{t-s}f)\in D(Q)$ and $P_s\zeta\in D(L)\subset D(Q).$ Hence, using Green's formula \eqref{e:green formula} in Lemma \ref{l:green formula}, we obtain that
  \begin{equation}\label{e:Gs2}G'(s)=-2\sum_{x}\Gamma(P_{t-s}f,\Delta P_{t-s}f)(x)P_s\zeta(x)m(x)-\sum_{x}\Gamma(\Gamma(P_{t-s}f),P_{s}\zeta)(x)m(x).\end{equation}

  We claim that for any $0\leq h\in D(Q),$ \begin{eqnarray}\label{e:Gs3}
    &&-2\sum_{x}\Gamma(P_{t-s}f,\Delta P_{t-s}f)(x)h(x)m(x)-\sum_{x}\Gamma(\Gamma(P_{t-s}f),h)(x)m(x)\\
    &\geq& 2K\sum_{x}\Gamma(P_{t-s}f)h(x)m(x).\nonumber
  \end{eqnarray} Once this claim is verified, by applying $h=P_{s}\zeta$ in \eqref{e:Gs3} and the self-adjointness of operators $P_t,$ we can prove the theorem. This claim can be proved by a density argument. Firstly, the $\CCD(K,\infty)$ condition yields that \eqref{e:Gs3} holds for $0\leq h\in C_0(V)$: In fact, by Green's formula for $h\in C_0(V)$,
  \begin{eqnarray*}
    &&-2\sum_{x}\Gamma(P_{t-s}f,\Delta P_{t-s}f)(x)h(x)m(x)-\sum_{x}\Gamma(\Gamma(P_{t-s}f),h)(x)m(x)\\
    &=&-2\sum_{x}\Gamma(P_{t-s}f,\Delta P_{t-s}f)(x)h(x)m(x)+\sum_{x}\Delta(\Gamma(P_{t-s}f))(x)h(x)m(x)\\
    &\geq& 2K\sum_{x}\Gamma(P_{t-s}f)h(x)m(x),
  \end{eqnarray*} where in the last inequality we used the $\CCD(K,\infty)$ condition.

For general $0\leq h\in D(Q),$ set $h_k=h\eta_k$ where $\{\eta_k\}$ is defined in \eqref{d:complete}. It is obvious that $0\leq h_k\in C_0(V).$ Note that Lemma \ref{lem:energy estimate} and Theorem~\ref{thm:w12} yield that $\Gamma(P_{t-s}f,\Delta P_{t-s}f), \Gamma(P_{t-s}f)\in \ell^1_m$ and  $\Gamma(P_{t-s}f)\in D(Q).$ Hence applying \eqref{e:Gs3} for $h_k,$ passing to the limit, $k\to \infty,$ we prove the theorem.

\end{proof}

Now we can prove the gradient bounds of heat semigroups under the $\CCD(K,\infty)$ condition.
\begin{proof}[Proof of Theorem \ref{thm:main}]
  $(a)\Rightarrow (b):$ Using the same notation as in Theorem \ref{thm:monotone}, we get
  $$G'(s)\geq 2KG(s).$$ Hence $G(s)\geq e^{2Ks}G(0).$ Since $P_s$ is a self-adjoint operator on $\ell^2_m,$
  $$G(s)=\sum_{x\in V}P_s(\Gamma(P_{t-s}f))(x)\zeta(x)m(x).$$ By choosing delta functions, such as $\zeta(x)=\delta_y(x)$ ($y\in V$), we prove the theorem.

  $(b)\Rightarrow (a):$ Fix a vertex $x\in V.$ By $(b),$ $$F(t):=e^{-2Kt}P_t(\Gamma (f))(x)-\Gamma(P_t f)(x)\geq 0.$$ It is easy to see that $F(t)$ is differentiable and $F'(0)\geq 0.$ Note that $$\left.\frac{d}{dt}\right|_{t=0}P_t(\Gamma (f))(x)=\Delta P_t(\Gamma(f))(x)|_{t=0}=\Delta (\Gamma (f))(x).$$
  Since the graph is locally finite, $$\left.\frac{d}{dt}\right|_{t=0}\Gamma(P_t f)(x)=2\Gamma(P_tf,\Delta P_tf)(x)|_{t=0}=2\Gamma(f,\Delta f)(x).$$
  This proves the assertion by using $F'(0)\geq 0.$

  $(b)\Leftrightarrow(c):$ This follows from a density argument.
\end{proof}


Now we are ready to prove the analogue to Yau's result \cite{Yau78} on graphs.

 \begin{proof}[Proof of Theorem~\ref{thm:main theorem2}]
   It suffices to prove that $P_t \mathds{1}=\mathds{1}$ where $\mathds{1}$ is the constant function $1$ on $V.$ By completeness, let $\eta_k\in C_0(V)$ satisfy \eqref{d:complete}. The dominated convergence theorem yields that $P_t \eta_k\to P_t \mathds{1}$ pointwise. By the local finiteness of the graph, for any $x\in V$ and $t>0,$
   \begin{eqnarray*}
     \Gamma(P_t\mathds{1})(x)&=&\lim_{k\to \infty} \Gamma(P_t \eta_k)(x)\leq \liminf_{k\to\infty}e^{-2K t}P_t (\Gamma(\eta_k))(x)\\
     &\leq &\liminf_{k\to\infty}e^{-2Kt}\cdot\frac{1}{k}=0.
   \end{eqnarray*} This means that for any $t>0,$ $P_t\mathds{1}$ is a constant function on $V.$ Since the function $P_t\mathds{1}$ is continuous in $t$ pointwise and $P_0\mathds{1}=\mathds{1},$ we get $P_t\mathds{1}=\mathds{1}$ for any $t>0.$ This proves the theorem.
 \end{proof}



{\bf Acknowledgements.} This work was done when the authors were visiting the Shanghai Center for Mathematical Sciences, Fudan University in Summer 2014. They acknowledge the support from SCMS.

B. H. is supported by NSFC, grant no. 11401106. Y. L. is supported by NSFC, grant no. 11271011, the Fundamental Research Funds for the Central Universities and the Research Funds of Renmin University of China($11$XNI$004$).

\bibliography{Stoch}
\bibliographystyle{alpha}

\end{document}